\documentclass[12pt,a4paper]{article}
\usepackage{parskip}
\usepackage{graphicx} 
\usepackage{amsfonts,amssymb,amsthm,amsmath,amscd}
\usepackage{mathtools}
\usepackage[margin=1.0in]{geometry}
\usepackage{enumitem}
\usepackage{tabularray}
\usepackage{tikz}
\usepackage{tikzstyle}
\usetikzlibrary{intersections}
\usepackage[x11names]{xcolor}
\usepackage[table]{xcolor}
\usepackage{caption}
\usepackage{thmtools} 
\usepackage{url}
\usepackage[linktoc=all,hidelinks,colorlinks,unicode=true]{hyperref} 
\usepackage[capitalise,compress,nameinlink,noabbrev]{cleveref} 
\hypersetup{linkcolor={blue!70!black},citecolor={blue!70!black},urlcolor={blue!70!black}}

\declaretheorem[name=Theorem,numberwithin=section]{theorem}
\declaretheorem[name=Lemma,sibling=theorem]{lemma}

\declaretheorem[name=Observation,sibling=theorem]{observation}
\declaretheorem[name=Conjecture,sibling=theorem]{conjecture}



\DeclareMathOperator{\supp}{supp}

\newif\ifcomments
\commentstrue    

\usepackage[backend=biber, maxbibnames=99]{biblatex}
\title{Sharp bounds for covering with large cliques and independent sets} 

\author{Veronica Bitonti\footnote{Mathematical Institute, University of Oxford, United Kingdom,\\ \texttt{\{veronica.bitonti, emma.hogan, tommy.walkermackay\}@maths.ox.ac.uk}}\protect\footnotemark[1] \footnote{Supported by EPSRC grant EP/Z534870/1} \qquad Emma Hogan\protect\footnotemark[1] \footnote{Supported by EPSRC grant EP/X013642/1.} \qquad Tommy Walker Mackay\protect\footnotemark[1]\qquad}
\date{\today}

\begin{document}

\maketitle

\begin{abstract}

Let $n(k_1, k_2)$ be the least integer $n$ such that there exists a graph on $n$ vertices in which every vertex is contained in both a clique of size $k_1$ and an independent set of size $k_2$. Recently, Feige and Pauzner showed that ${n(k, k) \geq 4k-O(k^\frac{2}{3})}$, and conjectured that $n(k,k)=4k-4$. We prove this conjecture, and also establish the optimal lower bound in the more general case where $k_1$ and $k_2$ are arbitrary. We further consider the generalisation of the problem to $r$-edge-coloured complete graphs in which every vertex is contained in a size-$k$ monochromatic clique of each colour, and obtain upper and lower bounds on the size of such graphs.
\end{abstract}

\section{Introduction}\label{sec:introduction}

A central theme in extremal graph theory is understanding the extent to which local constraints force the existence of global structure. Motivated by a recent conjecture of Feige and Pauzner~\cite{FP2025}, in this paper, we study graphs in which every vertex is contained in both a large clique and a large independent set.

Let $G$ be a graph on $n$ vertices, and say that $G$ is \emph{$k$-enabling} if every vertex of $G$ is contained in both a clique of size $k$ and an independent set of size $k$. Let $k(n)$ be the maximum value of $k$ for which a $k$-enabling graph on $n$ vertices exists. Motivated by the study of secret sharing schemes (see, for example, \cite{SCSS} and \cite{HTSAS}), Feige and Pauzner proved the following upper bound for $k(n)$:
\begin{equation*}
    k(n) \leq \frac{n}{4}+O\left(n^\frac{2}{3}\right),
\end{equation*}
and made a conjecture for the exact value of $k(n)$.
\begin{conjecture}[\cite{FP2025}, Conjecture 1]\label{conj:feige_conjecture}
    For all $n > 0$, $k(n)=\left \lfloor{\frac{n}{4}}\right \rfloor+1$.
\end{conjecture}
Feige and Pauzner further demonstrated that this conjecture is tight, if true, using the following construction for a certain blow-up of $P_4$. Consider an $n$-vertex graph with vertex set $V$, where $V$ is partitioned as equally as possible into $V_1 \sqcup V_2 \sqcup V_3 \sqcup V_4$. Let $V_2$ and $V_3$ each induce a clique, and also include every edge between $V_1$ and $V_2$, $V_2$ and $V_3$, and $V_3$ and $V_4$, and include no further edges. It is easily verified that each vertex in this graph lies in both a clique and an independent set of size $\left \lfloor{\frac{n}{4}}\right \rfloor+1$.

The main result of this paper is a proof of a generalisation of \cref{conj:feige_conjecture}. Let $G$ be a graph on $n$ vertices, and for $k_1, k_2 \geq 1$, say $G$ is $(k_1, k_2)$-enabling if every vertex of $G$ is contained in both a clique of size $k_1$ and an independent set of size $k_2$. Define $n(k_1, k_2)$ to be the least value of $n$ such that there exists a $(k_1, k_2)$-enabling graph on $n$ vertices. Note that in the case where $k_1 = 1$, the trivial bound $n \geq k_2$ is achieved by the empty graph on $k_2$ vertices, and so $n(1,k_2) = k_2$. Similarly, via the complete graph on $k_1$ vertices, $n(k_1,1) = k_1$. Since Feige and Pauzner showed that $k(n)$ is a non-decreasing function of $n$, and clearly $n(k,k)$ is a non-decreasing function of $k$, Feige and Pauzner's result is equivalent to $n(k,k) \geq 4k - O(k^{\frac{2}{3}})$, and \cref{conj:feige_conjecture} is equivalent to $n(k,k) = 4k-4$. Our main result is as follows.

\begin{theorem} \label{thm:main} Let $k_1, k_2 \geq 2$. We have
    \begin{equation*}
        n(k_1, k_2) \geq \left(\sqrt{k_1-1} + \sqrt{k_2-1}\right)^2,
    \end{equation*}
    with equality whenever $\left(\sqrt{k_1-1} + \sqrt{k_2-1}\right)^2$ is an integer.
\end{theorem}
In particular, when $k_1 = k_2 = k$, \cref{thm:main} immediately establishes \cref{conj:feige_conjecture}. 

We further generalise the problem of Feige and Pauzner to the setting of $r$-edge-coloured complete graphs in which every vertex is contained in a monochromatic clique of each colour. Let $G$ be an $r$-edge-coloured complete graph on $n$ vertices, and call $G$ \emph{$r$-coloured $k$-enabling} if every vertex of $G$ is contained in a monochromatic size-$k$ clique of each of the $r$ colours. Define $n_r(k)$ to be the least integer $n$ such that there exists an $r$-coloured $k$-enabling graph on $n$ vertices. We prove the following bounds on $n_r(k)$.

\begin{theorem}\label{thm:multicol}
   For all positive integers $k, r \geq 2$, we have $$2rk - 2r(r-1) \leq n_r(k) \leq 2r(k-1).$$
\end{theorem}

Our proof strategies for the lower bounds in \cref{thm:main,thm:multicol} use both probabilistic and optimisation techniques. In \cref{sec:main} we we associate probability measures to the sets of large cliques and independent sets of some $(k_1,k_2)$-enabling graph $G$, and apply a linear programming duality argument that shows these structures must be in some sense ``spread out'' across the vertices of $G$. This yields the lower bound of \cref{thm:main}, and the upper bound is achieved by explicit construction. In \cref{sec:morecolours} we generalise the ideas of \cref{sec:main} to the multi-colour setting, establishing the lower bound in \cref{thm:multicol}. We again provide an explicit construction that establishes the upper bound.

We further note that neither the lower bound nor the upper bound in \cref{thm:multicol} are, in general, tight. In the multicolour setting, a trivial lower bound of $n_r(k) \geq r(k-1)+1$ may be obtained by observing that each vertex is contained in $r$ size-$k$ monochromatic cliques of distinct colours. When $k \leq 2r-3$, the lower bound in \cref{thm:multicol} is strictly worse than this trivial lower bound. We conclude the paper by exhibiting an additional construction that demonstrates that for certain values of $r$ and $k$ where $k$ is not much larger than $r$, the upper bound in \cref{thm:multicol} can be further reduced. Further refinement of both the upper and lower bounds would be of interest. When $k$ is much larger than $r$, we conjecture that the upper bound in \cref{thm:multicol} is tight.

\begin{restatable}{conjecture}{conjMain} \label{conj:main}
    Fix $r \geq 2$. There exists some $k_0 \geq r$ such that for all $k \geq k_0$, we have \begin{equation*}
        n_r(k) = 2r(k-1).
    \end{equation*}
\end{restatable}

This work relates to existing research in local Ramsey theory, which studies properties of graphs in which the collection of cliques or independent sets of a particular size is, in some sense, locally dense. Alon, Buci{\'c} and Sudakov \cite{alon2021large} consider graphs on $n$ vertices and find an upper bound on the minimal possible $m$ such that every set of $m$ vertices contains both a clique and an independent set of size $\log n$. Buci{\'c} and Sudakov \cite{bucic2023large} find a lower bound on the independence number of a graph on $n$ vertices in which every set of $7$ vertices contains an independent set of size $3$. By contrast, we require that every vertex of our graph is contained in both a clique and an independent set of some given size, but do not otherwise restrict the subset of the vertices that must contain this clique and independent set.

Throughout this paper we will use the following notation. Let $G = (V, E)$ be an edge-coloured graph. For any subset $S \subseteq V$, we write $G[S]$ for the induced subgraph of $G$ with vertex set $S$. When $G[S]$ is a monochromatic clique, we will sometimes identify $G[S]$ with $S$. For disjoint subsets $A,B \subseteq V$, we denote by $G[A,B]$ the bipartite graph induced between vertex parts $A$ and $B$ of $G$. Finally, we adopt the standard convention that $[n] = \{1,\dots,n\}$.

\section{Proof of \cref{thm:main}} \label{sec:main}

Let $G = (V,E)$ be a $(k_1, k_2)$-enabling graph on $n$ vertices, with $k_1,k_2 \geq 2$. Observe that, by considering edges of $G$ to be coloured red and non-edges of $G$ to be edges coloured blue, we may consider $G$ to be a $2$-edge-coloured complete graph in which every vertex is contained in both a red clique of size $k_1$ and a blue clique of size $k_2$. For consistency with the $r$-colour setting addressed in \cref{sec:morecolours}, we will prove \cref{thm:main} using this language. We remark that at no point in this section will we rely on the fact that every edge of $G$ must be coloured either red or blue -- indeed, our proof applies to any $2$-edge-coloured graph on $n$ vertices. This will allow us to reuse many of the arguments in \cref{sec:morecolours} when we address the case with more than $2$ colours. We begin by introducing some notation that will be used throughout this section.

For each vertex $v \in V$, let $R(v)$ and $B(v)$ be subsets of $V$ such that $|R(v)| = k_1$, $G[R(v)]$ is a red clique containing $v$, $|B(v)| = k_2$ and $G[B(v)]$ is a blue clique containing $v$. If there are multiple choices for either $R(v)$ or $B(v)$, we fix one choice arbitrarily. We will sometimes want to consider the collection of all the red size-$k_1$ cliques (or blue size-$k_2$ cliques) that have been fixed for the vertices in $V$. To this end, define the families
\begin{flalign*}
    \mathcal{R} &\coloneq \{R(v): v \in V\}, \text{ and } \\
    \mathcal{B} &\coloneq \{B(v): v \in V\}.
\end{flalign*}
The proof of \cref{thm:main} relies on the construction of certain probability measures on $\mathcal{R}$ and $\mathcal{B}$. We will use the following notation throughout the argument. Let $\mathcal{P}$ denote the set of all probability measures on $V$. Let
\begin{flalign*}
    \delta_1 &\coloneq \max_{\lambda \in \mathcal{P}}\min_{R \in \mathcal{R}} \lambda(R), \text{ and}\\
    \delta_2 &\coloneq \max_{\lambda \in \mathcal{P}}\min_{B \in \mathcal{B}} \lambda(B),
\end{flalign*}
noting that $\delta_1$ and $\delta_2$ exist because $\mathcal{P}$ is compact. Then choose $\lambda_1$ and $\lambda_2$ to be arbitrary measures satisfying 
\begin{flalign*}
    \min_{R \in \mathcal{R}} \lambda_1(R) &= \delta_1, \text{ and }\\
    \min_{B \in \mathcal{B}} \lambda_2(B) &= \delta_2.
\end{flalign*}
In other words, $\delta_1$ is the minimum measure of a red clique in $\mathcal{R}$ according to the optimised measure $\lambda_1$, and $\delta_2$ is the minimum measure of a blue clique in $\mathcal{B}$ according to the optimised measure $\lambda_2$. 

Our strategy for proving \cref{thm:main} is as follows. We begin by establishing lower bounds for each of $\delta_1$ and $\delta_2$, as well as an upper bound on their sum. This will constrain the optimisation performed in later steps. Next, in \cref{lem:mu_existence}, we establish the existence of measures $\mu_1$ on $\mathcal{R}$ and $\mu_2$ on $\mathcal{B}$ so that for every vertex in $V$, the $\mu_1$-measure of the red cliques containing $v$ is at most $\delta_1$, and the $\mu_2$-measure of the blue cliques containing $v$ is at most $\delta_2$. The existence of $\mu_1$ and $\mu_2$ is established by checking the feasible region of a relevant linear programming problem. Roughly speaking, this establishes that the red and blue cliques in $G$ have to be somewhat ``spread out'' over the vertex set. Next, in \cref{lem:delta_dep}, we use the existence of $\mu_1$ and $\mu_2$ to establish a bound on $n(k_1,k_2)$ in terms of $\delta_1$ and $\delta_2$. Finally, in \cref{lem:lowerbound}, we complete the proof of the lower bound on $n(k_1,k_2)$ by optimising the bound from \cref{lem:delta_dep} in terms of the constraints on $\delta_1$ and $\delta_2$. We complete the proof of \cref{thm:main} by providing a construction that shows this bound is sharp.

We begin with trivial lower bounds on $\delta_1$ and $\delta_2$.
\begin{observation} \label{rmk:delta_pos}
For all choices of $\mathcal{R}$ and $\mathcal{B}$, we have
    \begin{align*}
        \delta_1 &\geq \frac{k_1}{n}, \text{ and} \\
        \delta_2 &\geq \frac{k_2}{n}.
    \end{align*}
    In particular, $\delta_1$ and $\delta_2$ are both strictly positive.
\end{observation}
\begin{proof}
    Let $\lambda$ be the uniform measure on $V$. Then we have
    \begin{flalign*}
        \min_{R \in \mathcal{R}} \lambda(R) &=  \frac{k_1}{n}, \text{ and} \\
        \min_{B \in \mathcal{B}} \lambda(B) &= \frac{k_2}{n},
    \end{flalign*}
    and the result follows from the definitions of $\delta_1$ and $\delta_2$.
\end{proof}

We now aim to show that $\delta_1 + \delta_2 \leq 1$. Our strategy will be to randomly sample vertices $s$ and $t$ from $V$ according to the probability measures $\lambda_1$ and $\lambda_2$ respectively, and bound $\delta_1$ and $\delta_2$ in terms of the probability that $t \in B(s)$. We begin with a lemma that will help us to handle the case where $s = t$.
\begin{lemma} \label{lem:big_red_clique_2}
    Suppose that $\delta_1 > \frac{1}{2}$ and let $S = \supp(\lambda_1)$. Then $G[S]$ is a red clique.
\end{lemma}

\begin{proof}
    Suppose for a contradiction that $xy \in G[S]$ is not red. Let $\ell = \min\{\lambda_1(x), \lambda_1(y)\}$, and let $\lambda' \in \mathcal{P}$ be given by
    \begin{equation*}
        \lambda'(v) = \begin{cases}
            \alpha(\lambda_1(v) - \ell), &\text{if } v \in \{x,y\};\\
            \alpha\lambda_1(v), &\text{otherwise},
        \end{cases}
    \end{equation*}
    where $\alpha = \frac{1}{1-2\ell}$. Now taking $R \in \mathcal{R}$, we have
    \begin{flalign*}
        \lambda'(R) - \lambda_1(R)&= \sum_{v \in R} (\lambda'(v) - \lambda_1(v)) \\
        &= (\alpha-1)\left(\sum_{v \in R} \lambda_1(v)\right) - \alpha \ell |R \cap \{x,y\}| \\
        &\geq (\alpha - 1)\left(\sum_{v \in R} \lambda_1(v)\right) - \alpha \ell,
    \end{flalign*}
    where the inequality follows from the fact that $xy$ is not a red edge and therefore cannot be contained in a red clique. Thus, by assumption,
    \begin{flalign*}
        \lambda'(R) - \lambda_1(R)& > (\alpha - 1) \cdot \frac{1}{2} - \alpha \ell \\
        &= \frac{1}{2}\left(\frac{1}{1-2\ell}-1\right) - \frac{\ell}{1-2\ell} = 0,
    \end{flalign*}
    and so \begin{equation*}
        \min_{R \in \mathcal{R}}\lambda_1(R) < \min_{R \in \mathcal{R}} \lambda'(R),
    \end{equation*}
    contradicting the definition of $\lambda_1$. It follows that $G[S]$ is a red clique.
\end{proof}
The above enables us to conclude the following:
\begin{lemma} \label{lem:delta_sum}
    $\delta_1 + \delta_2 \leq 1$.
\end{lemma}

\begin{proof}
    Let $S := \supp(\lambda_1)$ as in \cref{lem:big_red_clique_2}, and suppose for a contradiction that ${\delta_1 + \delta_2 > 1}$. We may assume without loss of generality that $\delta_1 > \frac{1}{2}$, and so it follows from \cref{lem:big_red_clique_2} that $G[S]$ is a red clique. 
    
    Sample a vertex $s$ of $G$ according to the probability distribution $\lambda_1$, and independently sample a vertex $t$ according to the probability distribution $\lambda_2$. That is, set $${\mathbb{P}[(s, t) = (s', t')] = \lambda_1(s')\lambda_2(t')}.$$ We consider $\mathbb{P}(t \in B(s))$, which we will show is bounded below by $\delta_2$, and above by $1 - \delta_1$, from which the result follows.
    
    Beginning with the lower bound, we have
    \begin{flalign*}
        \mathbb{P}(t \in B(s))&= \sum_{s \in V} \lambda_1(s) \lambda_2(B(s)) \\
        &\geq \sum_{s \in V} \lambda_1(s) \min_{v \in V} \lambda_2(B(v)) \\
        &= \min_{v \in V} \lambda_2(B(v)) \\
        &= \delta_2,
    \end{flalign*}
    as desired. For the upper bound, consider $\mathbb{P}(t \in B(s) \mid t = t')$. If $t' \not\in S$, then $t' \in B(s)$ and $s \in R(t')$ are disjoint events, since in the first case $t's$ is blue and in the second $t's$ is red. Hence, if $t' \not\in S$, we have
    \begin{flalign*}
        \mathbb{P}(t \in B(s) \mid t = t') &\leq 1 - \mathbb{P}(s \in R(t')) \\
        &= 1 - \lambda_1(R(t')) \\
        &\leq 1 - \min_{R \in \mathcal{R}} \lambda_1(R) \\
        &= 1 - \delta_1.
    \end{flalign*}
    If $t' \in S$, then it follows from \cref{lem:big_red_clique_2} that $t' \in B(s)$ only if $s = t'$. In this case, take a vertex $v$ such that $t'v$ is blue (which must exist because $t'$ is contained in a blue clique of size $k_2 \geq 2$), and note that $R(v) \cap \{t'\} = \emptyset$. Thus,
    \begin{flalign*}
        \mathbb{P}(t \in B(s) \mid t = t') &\leq \lambda_1(t') \\
        &= 1 - \lambda_1(V \setminus \{t'\}) \\
        &\leq 1 - \lambda_1(R(v)) \\
        &\leq 1 - \min_{R \in \mathcal{R}} \lambda_1(R) \\
        &= 1 - \delta_1.
    \end{flalign*}
    It follows that for all values of $t'$, we have $\mathbb{P}(t \in B(s) \mid t = t') \leq 1 - \delta_1$, and therefore $\mathbb{P}(t \in B(s)) \leq 1 - \delta_1$, and $\delta_1 + \delta_2 \leq 1$, as desired. This completes the proof.
\end{proof}

We now apply linear optimisation to construct probability measures $\mu_1$ and $\mu_2$ on the set of red cliques $\mathcal{R}$ and the set of blue cliques $\mathcal{B}$ respectively.

\begin{lemma} \label{lem:mu_existence}
    There exist probability measures $\mu_1$ on $\mathcal{R}$ and $\mu_2$ on $\mathcal{B}$ such that for every $v \in V$, we have
    \begin{align*}
        &\mu_1(\{X \in \mathcal{R}: v \in X\}) \leq \delta_1, \text{ and } \\
       &\mu_2(\{X \in \mathcal{B}: v \in X\}) \leq \delta_2.
    \end{align*}
\end{lemma}

\begin{proof}
    We establish the existence of a suitable $\mu_1$, as the proof for $\mu_2$ follows similarly. Consider the following linear optimisation problem.
    \begin{center}
    \setlength{\tabcolsep}{15pt} 
    \renewcommand{\arraystretch}{1.2} 
    \begin{tabular}{l l l l}
        $\mathbf{P}$: & maximise & $\sum_{X \in \mathcal{R}} \mu(X)$ \\
        & subject to & $\sum_{X \in \mathcal{R}: v \in X} \mu(X) \leq \delta_1$ & $\forall v$ \\
        & and & $\mu(X) \geq 0$ & $\forall X \in \mathcal{R}$. 
    \end{tabular}
    \end{center}

    We will show that if $\mu^{*}$ is a maximiser for $\mathbf{P}$, then $\sum_{X \in \mathcal{R}} \mu^{*}(X) \geq 1$. It follows by scaling each $\mu^*(X)$ term by the appropriate constant that a suitable probability measure $\mu_1$ exists. Note that $\mathbf{P}$ is a bounded linear program (since for all $X \in \mathcal{R}$, $\mu(X)$ is bounded above by $\delta_1$) and that the feasible region of $\mathbf{P}$ is non-empty. It follows that $\mathbf{P}$ has an optimal solution, and so by the strong duality theorem its optimal solution is equal to the optimal solution of the dual problem
    \begin{center}
    \setlength{\tabcolsep}{15pt} 
    \renewcommand{\arraystretch}{1.2} 
    \begin{tabular}{l l l l}
        $\mathbf{D}$: & minimise & $\delta_1 \sum_{v \in V} \lambda(v)$ \\
        & subject to & $\sum_{v \in X} \lambda(v) \geq 1$ & $\forall X \in \mathcal{R}$\\
        & and & $\lambda(v) \geq 0$ & $\forall v$.
    \end{tabular}
    \end{center}
    It therefore suffices to show that the optimal solution of $\mathbf{D}$ is at least $1$. Now, let $\lambda^*$ be a minimiser for $\mathbf{D}$, and let $\nu$ be the probability measure on $V$ defined by letting $\nu(v) := \frac{\lambda^{*}(v)}{\sum_{v' \in V} \lambda^{*}(v')}$ for each $v \in V$. Then for each $X$ in $\mathcal{R}$, we have \begin{flalign*}
        \sum_{v \in X} \nu(v) &\geq \frac{1}{\sum_{v' \in V} \lambda^{*}(v')}
    \end{flalign*}
    and so,
    \begin{flalign*}    
        \min_{X \in \mathcal{R}} \nu(X) &\geq \frac{1}{\sum_{v' \in V}\lambda^{*}(v')}.
    \end{flalign*}
    Optimising $\nu$ over the collection of all possible probability measures $\mathcal{P}$, we have
    \begin{flalign*}
        \frac{1}{\sum_{v' \in V} \lambda^{*}(v')} \leq \min_{X \in \mathcal{R}} \nu(X) \leq \delta_1,
    \end{flalign*}
    by the definition of $\delta_1$. It follows that $\delta_1 \sum_{v \in V} \lambda^{*}(v) \geq 1$, as required. Hence, the primal problem $\mathbf{P}$ satisfies $\sum_{X \in \mathcal{R}} \mu^*(X) \geq 1$, and we can take $\mu_1(X) = \frac{\mu^{*}(X)}{\sum_{Y \in \mathcal{R}} \mu^{*}(Y)}$ to be a suitable probability measure on $\mathcal{R}$. Interchanging the roles of the red and blue cliques as well as $\delta_1$ and $\delta_2$ in this argument similarly establishes the existence of a suitable $\mu_2$, completing the proof.
\end{proof}

We are now able to bound $n$ in terms of $\delta_1$ and $\delta_2$. 
\begin{lemma} \label{lem:delta_dep}
    Let $\delta_1' \geq \delta_1$ and $\delta_2' \geq \delta_2$. Then \begin{equation*}
        n \geq \frac{k_1}{\delta_1'} + \frac{k_2}{\delta_2'} - \frac{1}{\delta_1'\delta_2'}.
    \end{equation*}
\end{lemma}

\begin{proof}
    Let $\mu_1$ and $\mu_2$ be probability measures on $\mathcal{R}$ and $\mathcal{B}$ as described in \cref{lem:mu_existence}. For each vertex $v$ in $V$, we write 
    \begin{flalign*}
        \mu_1(v) &:= \mu_1(\{X \in \mathcal{R}: v \in X\}), \text{ and} \\
        \mu_2(v) &:= \mu_2(\{X \in \mathcal{B}: v \in X\}).
    \end{flalign*}
    Summing over every vertex in $V$, we have
    \begin{flalign*}
        \sum_{v \in V} \mu_1(v) &= \sum_{v \in V} \sum_{X \in \mathcal{R}: v \in X} \mu_1(X) \\
        &= \sum_{X \in \mathcal{R}} \mu_1(X) |X| \\
        &= k_1 \sum_{X \in \mathcal{R}}\mu_1(X)\\ 
        &= k_1,
    \end{flalign*}
    and similarly, $\sum_{v \in V} \mu_2(v) = k_2$. Furthermore, \begin{flalign*}
        \sum_{v \in V}\mu_1(v)\mu_2(v) &= \sum_{v \in V}\sum_{X \in \mathcal{R}: v \in X} \sum_{Y \in \mathcal{B}: v \in Y} \mu_1(X) \mu_2(Y) \\
        &= \sum_{X \in \mathcal{R}} \sum_{Y \in \mathcal{B}}\mu_1(X) \mu_2(Y) |X \cap Y| \\
        &\leq \sum_{X \in \mathcal{R}} \sum_{Y \in \mathcal{B}}\mu_1(X) \mu_2(Y) \\
        &= 1,
    \end{flalign*}
    where the inequality follows from the fact that a red clique and a blue clique intersect at at most $1$ vertex. For each $i \in \{1,2\}$, let $\alpha_i = \frac{1}{\delta_i'}$ (noting that by \cref{rmk:delta_pos}, we have that $\delta_i$ is positive). Then by \cref{lem:mu_existence}, for all $v \in V$, we have $\alpha_i\mu_i(v) \leq \alpha_i \delta_i \leq \alpha_i \delta_i' = 1$, and so \begin{flalign*}
        0 &\leq \sum_{v \in V} (1 - \alpha_1 \mu_1(v))(1- \alpha_2 \mu_2(v)) \\
        &= \sum_{v \in V} (1 - \alpha_1 \mu_1(v) - \alpha_2 \mu_2(v) + \alpha_1\alpha_2\mu_1(v)\mu_2(v)) \\
        &= n - \alpha_1 \sum_{v \in V}\mu_1(v) - \alpha_2 \sum_{v \in V}\mu_2(v) + \alpha_1\alpha_2\sum_{v \in V}\mu_1(v) \mu_2(v) \\
        &\leq n - \alpha_1k_1 - \alpha_2k_2 + \alpha_1\alpha_2
    \end{flalign*}
    and so $n \geq \frac{k_1}{\delta_1'} + \frac{k_2}{\delta_2'} - \frac{1}{\delta_1'\delta_2'}$, as required.
\end{proof}

Finally, we optimise this bound, establishing the lower bound on $n(k_1, k_2)$ given in \cref{thm:main}.
\begin{lemma} \label{lem:lowerbound} For all $k_1, k_2 \geq 2$, we have
    \begin{equation*}
        n(k_1, k_2) \geq \left(\sqrt{k_1-1}+\sqrt{k_2-1}\right)^2.
    \end{equation*}
\end{lemma}
\begin{proof}
    By \cref{rmk:delta_pos}, we have that $\delta_1, \delta_2 > 0$, and by \cref{lem:delta_sum}, we have $\delta_1 + \delta_2 \leq 1$. It follows that we may choose some $\delta_1'$ and $\delta_2'$ such that $\delta_1' + \delta_2' = 1$, $\delta_1 \leq \delta_1' < 1$ and $\delta_2 \leq \delta_2' < 1$. Then by \cref{lem:delta_dep}, we have
    \begin{flalign*}
        n &\geq \frac{k_1}{\delta_1'} + \frac{k_2}{\delta_2'} - \frac{1}{\delta_1' \delta_2'} \\
        &= \frac{k_1-1}{\delta_1'} + \frac{k_2-1}{\delta_2'} \\
        &= \frac{(\sqrt{k_1-1})^2}{\delta_1'} + \frac{(\sqrt{k_2-1})^2}{1-\delta_1'} \\
        &\geq \left(\sqrt{k_1-1} + \sqrt{k_2-1}\right)^2
    \end{flalign*}
    by Sedrakyan's inequality. It follows that $n(k_1, k_2) \geq (\sqrt{k_1-1}+\sqrt{k_2-1})^2$. 
\end{proof}

To complete the proof of \cref{thm:main}, it now suffices to construct a $(k_1,k_2)$-enabling graph on $n = (\sqrt{k_1-1}+\sqrt{k_2-1})^2$ vertices whenever $(\sqrt{k_1-1}+\sqrt{k_2-1})^2 \in \mathbb{N}$. In the case $k_1 = k_2 = k$, our construction gives a family of extremal examples that generalises the $P_4$ blow-up construction of Feige and Pauzner.

\begin{lemma} \label{lem:eq_if_int}
    Suppose that \begin{equation*}
        n = \left(\sqrt{k_1-1}+\sqrt{k_2-1}\right)^2
    \end{equation*}
    is a positive integer. Then there exists a $(k_1, k_2)$-enabling graph $G$ on $n$ vertices. 
\end{lemma}
\begin{proof}
Note that $2\sqrt{k_1-1} \cdot \sqrt{k_2-1} = n - (k_1-1) - (k_2-1) \in \mathbb{N}$ and so $\sqrt{\frac{k_1-1}{k_2-1}} \in \mathbb{Q}$. Letting $g = \gcd(k_1 - 1, k_2 - 1)$, it follows that $\frac{k_1-1}{g}$ and $\frac{k_2-1}{g}$ are both square integers, so let 
\begin{flalign*}
    k_1 &= 1 + ga^2 \\
    k_2 &= 1 + gb^2,
\end{flalign*}
for positive integers $a$ and $b$. We construct a complete $2$-edge-coloured graph $G$ on vertex set $V = R \sqcup B$, where $|R| = ga(a+b)$ and $|B| = gb(a+b)$. Let $G[R]$ be a red clique, let $G[B]$ be a blue clique, and let the bipartite graph $G[R, B]$ be edge-coloured such that:
\begin{itemize}
    \item every vertex $v \in R$ has exactly $gab$ red neighbours in $B$, and
    \item every vertex $v \in B$ has exactly $ga^2$ red neighbours in $R$,
\end{itemize}
and the remaining edges in $G[R,B]$ are coloured blue. An example of such a graph when $k_1 = 3$ and $k_2=9$ is given in \cref{fig:2_colour_construction}. A suitable colouring can easily be seen to exist in general (there are exactly $g^2a^2b(a+b)$ red edges between the two parts, and in each vertex part, the vertices all see an equal share of them, so this is just a $(gab,ga^2)$-biregular bipartite graph). Furthermore, one may easily check that $G$ is a $(k_1,k_2)$-enabling graph on $n$ vertices, as desired.
\end{proof}

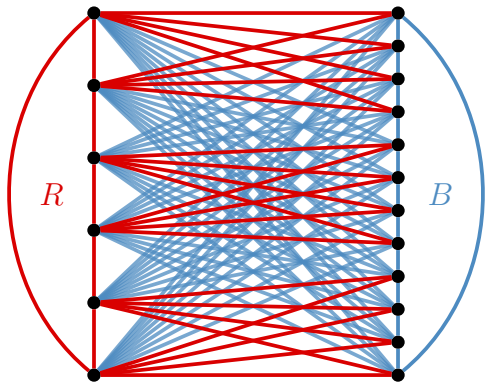
\begin{figure}[ht]
    \centering
    \begin{tikzpicture}[scale=0.8]
    \begin{scope}
        \foreach \x in {1,...,6}{
            \node[mycircle] (\x a) at (0,1.2*\x){};
        }
    \end{scope}
    
    \begin{scope}[xshift=5cm,yshift=0.6545cm]
        \foreach \x in {1,...,12}{
            \node[mycircle] (\x b) at (0,0.5454*\x){};
        }
    \end{scope}
    \foreach \x in {5,...,12}{
        \draw (1a) edge[color=myblue, opacity=0.75] (\x b);
        \draw (2a) edge[color=myblue, opacity=0.75] (\x b);
    }
    \foreach \x in {1,...,4,9,10,11,12}{
        \draw (3a) edge[color=myblue, opacity=0.75] (\x b);
        \draw (4a) edge[color=myblue, opacity=0.75] (\x b);
    }
    \foreach \x in {1,...,8}{
        \draw (5a) edge[color=myblue, opacity=0.75] (\x b);
        \draw (6a) edge[color=myblue, opacity=0.75] (\x b);
    }

    \foreach \x in {1,...,4}{
        \draw (1a) edge[color=myred] (\x b);
        \draw (2a) edge[color=myred] (\x b);
    }
    \foreach \x in {5,...,8}{
        \draw (3a) edge[color=myred] (\x b);
        \draw (4a) edge[color=myred] (\x b);
    }
    \foreach \x in {9,...,12}{
        \draw (5a) edge[color=myred] (\x b);
        \draw (6a) edge[color=myred] (\x b);
    }

    \foreach \x [evaluate=\x as \y using int(\x+1)] in {1,...,5} {
        \draw (\x a) edge[color=myred] (\y a);
    }
    \draw (1a) edge[color=myred, bend left=50] (6a);
    \node[color=myred] at (-0.7,4.2) {$R$};

    \foreach \x [evaluate=\x as \y using int(\x+1)] in {1,...,11} {
        \draw (\x b) edge[color=myblue] (\y b);
    }
    \draw (1b) edge[color=myblue, bend right=50] (12b);
    \node[color=myblue] at (5.7,4.2) {$B$};
\end{tikzpicture}
    \caption{An extremal construction for \cref{thm:main} in the case $k_1 = 3, k_2 = 9$.}
    \label{fig:2_colour_construction}
\end{figure} 

\cref{thm:main} now follows immediately from \cref{lem:lowerbound,lem:eq_if_int}. 

\section{Proof of \cref{thm:multicol}}\label{sec:morecolours}

We now consider the case where the complete graph on $n$ vertices is edge-coloured with $r \geq 2$ colours. Recall that $n_r(k)$ denotes the least integer $n$ such that there exists an $r$-coloured $k$-enabling graph on $n$ vertices. We prove \cref{thm:multicol}, that 
\begin{equation*}
     2rk - 2r(r-1) \leq n_r(k) \leq 2r(k - 1).
\end{equation*} 
Our proof strategy will be largely similar to the strategy for \cref{thm:main}. Suppose that $G = (V,E)$ is an $r$-edge-coloured complete graph on $n$ vertices for some $r \geq 2$, and that $G$ is $r$-coloured $k$-enabling. For each colour $i \in [r]$ and for each $v \in V$, let $C_i(v)$ be a monochromatic size-$k$ clique of colour $i$ containing $v$, and define $\mathcal{C}_i := \{C_i(v): v \in V\}$. Take $\mathcal{P}$ once again to be the set of probability measures on $V$, and define each $\lambda_i$ and $\delta_i$ analogously to $\lambda_1$ and $\delta_1$ in \cref{sec:main}. That is, set 
\[\delta_i \coloneq \max_{\lambda \in P} \min_{C \in \mathcal{C}_i} \lambda(C),\] 
and choose $\lambda_i$ to be any probability measure satisfying 
\[\min_{C \in \mathcal{C}_i}\lambda_i(C) = \delta_i.\] 
Finally, we define $\alpha_i \coloneq \frac{1}{\delta_i}$. Note that since $G$ is $r$-coloured $k$-enabling, we have that for any two colours, $i$ and $j$ in $[r]$, the subgraph of $G$ with edges coloured $i$ or $j$ must be $(k,k)$-enabling, and therefore all arguments from the previous section apply to this subgraph. Thus, by repeating the proof of \cref{lem:mu_existence} for each colour $i$, we may again construct probability measures $\mu_i$ on $\mathcal{C}_i$ satisfying the property that 
\begin{equation*}
    \mu_i(\{X \in \mathcal{C}_i: v \in X\}) \leq \delta_i.
\end{equation*}
We write $\mu_i(v)$ for $\mu_i(\{X \in \mathcal{C}_i: v \in X\})$, as in \cref{sec:main}, and define 
\begin{equation*}
    Y_i(v) \coloneq \alpha_i \mu_i(v).
\end{equation*}
Thus, for all $i \in [r]$ and for all $v \in V$, we have $Y_i(v) \in [0, 1]$. In order to produce a bound on $n$ similar to that in \cref{lem:delta_dep}, we will need the following inequality.
\begin{lemma} \label{lem:improved}
    For all real numbers $0 \leq x_1, \dots, x_m \leq 1$, we have 
    \begin{equation*}
        \sum_{1 \leq i < j \leq m} x_i x_j - \sum_{i=1}^{m} x_i + 1 \geq 0.
    \end{equation*}
\end{lemma}

\begin{proof} 
    Let $S = \sum_{i < j} x_i x_j - \sum_{i=1}^{m} x_i + 1$. Since $S$ is a multilinear function of variables in $[0,1]$, it follows that the minimum value of $S$ is achieved by some assignment $(x_1,\dots,x_m) \in \{0,1\}^m$. Let $s$ be the number of variables $x_i$ equal to $1$ at this minimiser. Then $S = \binom{s}{2} - s + 1$, which is non-negative at integer values of $s$, completing the proof.
\end{proof}

We will also use the following lower bound on $\sum_{i=1}^r \alpha_i$.
\begin{lemma} \label{lem:alpha_bar_bound}
    We have \begin{equation*}
        \sum_{i=1}^{r} \alpha_i \geq 2r.
    \end{equation*}
\end{lemma}
\begin{proof}
    By applying \cref{lem:delta_sum} with any distinct $i,j \in [r]$, we have that $\delta_i + \delta_j \leq 1$, and by \cref{rmk:delta_pos}, $\delta_i > 0$. Hence, $0 < \delta_i < 1$, and therefore 
    \begin{flalign*}
        \frac{1}{\delta_i} + \frac{1}{\delta_j} &\geq \frac{1}{\delta_i} + \frac{1}{1-\delta_i} \\
        &\geq 4.
    \end{flalign*}
    Recall that $\alpha_i  = \frac{1}{\delta_i}$, and hence $(r-1)\sum_i \alpha_i = \sum_{i < j}(\alpha_i + \alpha_j) \geq 4\binom{r}{2}$. Since we have $r \geq 2$, it follows that $\sum_{i=1}^{r} \alpha_i \geq 2r$, as desired.
\end{proof}

We now use \cref{lem:improved} and an adaptation of the proof of \cref{lem:delta_dep} to prove the lower bound.

\begin{lemma} \label{lem:multicol_lb}
    For all integers $r, k \geq 2$ we have $n_r(k) \geq 2rk - 2r(r-1).$
\end{lemma}

\begin{proof}
    Recall from the proof of \cref{lem:delta_dep} that, for all $i \neq j$, 
    \begin{flalign*}
        \sum_{v \in V} \mu_i(v) &= k\\
        \intertext{and}
        \sum_{v \in V} \mu_i(v)\mu_j(v) &\leq 1.
    \end{flalign*}
    By selecting a vertex $v \in V$ uniformly at random, it follows that
    \begin{flalign*}
        \mathbb{E}(Y_i) &= \frac{k\alpha_i}{n}, \\
        \intertext{and}
        \mathbb{E}(Y_i Y_j) &\leq \frac{\alpha_i \alpha_j}{n}.
    \end{flalign*}
    Let $S$ be a subset of $[r]$. Applying \cref{lem:improved} to the real values $\{Y_i \mid i \in S\}$ and taking expectations, we obtain 
    \begin{flalign*}
        \sum_{i < j \in S} \mathbb{E}(Y_i Y_j) - \sum_{i\in S} \mathbb{E}(Y_i) + 1 &\geq 0. 
    \end{flalign*}
    Substituting in the values for $\mathbb{E}(Y_i)$ and $\mathbb{E}(Y_iY_j)$, and multiplying by $n$, gives
    \begin{align*}
        \sum_{i < j \in S}\alpha_i\alpha_j - \sum_{i\in S} k\alpha_i + n &\geq 0.
    \end{align*}

Let $m \coloneq |S|$ and $\bar{\alpha}_S \coloneq \frac{1}{m} \sum_{i \in S}\alpha_i$. Then we have
\begin{flalign*}
    n &\geq k\sum_{i\in S}\alpha_i - \sum_{i < j \in S}\alpha_i\alpha_j  \\
    &= km \bar{\alpha}_S - \frac{1}{2} m^2 \bar{\alpha}_S^2 + \frac{1}{2} \sum_{i \in S}\alpha_i^2\\
    &\geq km \bar{\alpha}_S - \frac{1}{2} m^2 \bar{\alpha}_S^2 + \frac{1}{2}m \bar{\alpha}_S^2,
\end{flalign*}
where the final line follows from the fact that the root mean square of the real numbers $(\alpha_i : i \in S)$ is at least their arithmetic mean $\bar{\alpha}_S$. Assume now that for a fixed $m$, we choose $S$ to maximise $\bar{\alpha}_S$. Then $\bar{\alpha}_S \geq \bar{\alpha}$ (where $\bar{\alpha} \coloneq \frac{1}{r}\sum_{i=1}^{r} \alpha_i = \bar{\alpha}_{[r]}$), and so for all $0 \leq m \leq r$ we have
\begin{equation*}
    n \geq km \bar{\alpha} - \frac{m(m-1)}{2} \bar{\alpha}^2.
\end{equation*}
It remains to optimise this bound. Let 
\begin{align*}
    f_m(x) &\coloneq kmx - \frac{m(m-1)}{2}x^2, \\ 
    \intertext{and}
    f(x) &\coloneq \max_{0 \leq m \leq r} f_m(x).
\end{align*}
It follows that ${n \geq f(\bar\alpha)}$. We show that $f(x)$ is non-decreasing in $x$ for $x > 0$. Note that $f$ is continuous on $\mathbb{R}_{>0}$, and observe that $$f_m(x) - f_{m-1}(x) = kx - (m-1)x^2.$$ If $f_m'(x) < 0$, then we have $km - m(m-1)x < 0$, which implies that $m \neq 0$. Multiplying by $\frac{x}{m}$, this further implies that $$f_{m-1}(x) > f_{m}(x)$$ whenever $f_m'(x) < 0$. Hence, $f_m'(x) \geq 0$ in the region where $f(x) = f_m(x)$. Thus, $f$ is continuous and piecewise non-decreasing, and is therefore non-decreasing. By \cref{lem:alpha_bar_bound}, we have that $\bar{\alpha} \geq 2$, so
\begin{align*}
    f(\bar\alpha) &\geq \max_{0 \leq m \leq r} 2km - 2m(m-1)\\
    &\geq 2kr -2r(r-1),
\end{align*}

and hence $n_r(k) \geq 2kr-2r(r-1)$, as desired.
\end{proof}

To complete the proof of \cref{thm:multicol}, we now exhibit a construction for the upper bound.

\begin{lemma} \label{lem:conj_extrem}
    For all integers $r, k \geq 2$, we have $n_r(k) \leq 2r(k-1)$.
\end{lemma}
\begin{proof}
    Let $G$ be a complete graph with vertex set $V(G) = V_1 \sqcup V_2 \sqcup \dots \sqcup V_r$, where $|V_i| = 2(k-1)$ for each $i \in [r]$. For each $i$, let the edges in $G[V_i]$ be coloured $i$, and for each $1 \leq i \neq j \leq r$ let the bipartite graph $G[V_i, V_j]$ be coloured in such a way that each $v \in V_i$ has exactly $k-1$ neighbours in $V_j$ of colour $i$ and $k-1$ neighbours of colour $j$ (and vice-versa for vertices in $V_j$). An example in the case where $r=k=3$ is given in \cref{fig:multi_colour_construction1}. In general, such a colouring exists by the existence of $(k-1)$-regular bipartite graphs with vertex parts of size $2(k-1)$. It can easily be checked that $G$ is $r$-coloured $k$-enabling and has $2r(k-1)$ vertices, completing the proof.
\end{proof}

\cref{thm:multicol} now follows immediately from \cref{lem:multicol_lb,lem:conj_extrem}.

\begin{figure}[h]
    \centering
    \begin{tikzpicture}[scale=0.7]
    \begin{scope}[rotate=45]
        \foreach \x/\y in {0/1a,1/2a,2/3a,3/4a}{
        \node[mycircle] (\y) at (\x,0){};
        }
    \end{scope}

    \begin{scope}[xshift=6.24cm, yshift=2.12cm, rotate=-45]
        \foreach \x/\y in {0/1b,1/2b,2/3b,3/4b}{
            \node[mycircle] (\y) at (\x,0){};
            }
    \end{scope}

    \begin{scope}[xshift=2.68cm, yshift=-4.12cm]
        \foreach \x/\y in {0/1c,1/2c,2/3c,3/4c}{
        \node[mycircle] (\y) at (\x,0){};
        }
    \end{scope}

    \foreach \x [evaluate=\x as \y using int(\x+1)] in {1,2,3} {
        \draw (\x a) edge[color=myred] (\y a);
    }
    \draw (1a) edge[color=myred, bend left=60, looseness=1.2] (4a);
    \node[color=myred] at (0.75,1.35) {$V_1$};

    \foreach \x [evaluate=\x as \y using int(\x+1)] in {1,2,3} {
        \draw (\x b) edge[color=myblue] (\y b);
    }
    \draw (1b) edge[color=myblue, bend left=60, looseness=1.2] (4b);
    \node[color=myblue] at (7.6,1.35) {$V_2$};

    \foreach \x [evaluate=\x as \y using int(\x+1)] in {1,2,3} {
        \draw (\x c) edge[color=mygreen] (\y c);
    }
    \draw (1c) edge[color=mygreen, bend right=60, looseness=1.2] (4c);
    \node[color=mygreen] at (4.18,-4.6) {$V_3$};

    \foreach \x in {1,2}{
        \draw (1a) edge[color=myred, opacity=0.75] (\x b);
        \draw (2a) edge[color=myred, opacity=0.75] (\x b);
        \draw (1a) edge[color=myred, opacity=0.75] (\x c);
        \draw (2a) edge[color=myred, opacity=0.75] (\x c);
        \draw (3a) edge[color=myblue, opacity=0.75] (\x b);
        \draw (4a) edge[color=myblue, opacity=0.75] (\x b);
        \draw (3a) edge[color=mygreen, opacity=0.75] (\x c);
        \draw (4a) edge[color=mygreen, opacity=0.75] (\x c);

        \draw (1b) edge[color=myblue, opacity=0.75] (\x c);
        \draw (2b) edge[color=myblue, opacity=0.75] (\x c);
        \draw (3b) edge[color=mygreen, opacity=0.75] (\x c);
        \draw (4b) edge[color=mygreen, opacity=0.75] (\x c);
    }
    \foreach \x in {3,4}{
        \draw (1a) edge[color=myblue, opacity=0.75] (\x b);
        \draw (2a) edge[color=myblue, opacity=0.75] (\x b);
        \draw (1a) edge[color=mygreen, opacity=0.75] (\x c);
        \draw (2a) edge[color=mygreen, opacity=0.75] (\x c);
        \draw (3a) edge[color=myred, opacity=0.75] (\x b);
        \draw (4a) edge[color=myred, opacity=0.75] (\x b);
        \draw (3a) edge[color=myred, opacity=0.75] (\x c);
        \draw (4a) edge[color=myred, opacity=0.75] (\x c);

        \draw (1b) edge[color=mygreen, opacity=0.75] (\x c);
        \draw (2b) edge[color=mygreen, opacity=0.75] (\x c);
        \draw (3b) edge[color=myblue, opacity=0.75] (\x c);
        \draw (4b) edge[color=myblue, opacity=0.75] (\x c);
    }
\end{tikzpicture}
    \caption{Extremal construction for \cref{lem:conj_extrem} in the case $r=3$ and $k=3$.}
    \label{fig:multi_colour_construction1}
\end{figure}

\section{Conclusion}

While the lower bound found in \cref{thm:main} was sharp for infinitely many pairs $(k_1, k_2)$, we have already noted that the lower bound in \cref{thm:multicol} is not sharp in general. In fact, we conjecture that the upper bound in \cref{thm:multicol} is correct when $k$ is sufficiently large in terms of $r$.
\conjMain*

We conclude by observing that if \cref{conj:main} is correct, then the condition that $k$ is sufficiently large with respect to $r$ is necessary.
\begin{lemma}\label{lem:example2}
    There are infinitely many integers $k$ such that when $r=k+1$, we have $n_{r}(k) < 2r(k-1)$.
\end{lemma}
\begin{proof}
    Let $p$ be prime, and let $k=p$, $r = p+1$. Consider the complete graph $G$ on vertex set $V = \mathbb{Z}_p \times \mathbb{Z}_p$ with edge-colouring $c$ given by $c((x,y),(z,w)) = \frac{x-z}{y-w} \in \mathbb{Z}_p \cup \{\infty\}$ (where $c((x,y),(z,w)) = \infty$ if $y = w$ and division is otherwise interpreted in $\mathbb{Z}_p$). In the case $k=3$, this construction is demonstrated in \cref{fig:multi_colour_construction2}. Then $G$ is $r$-coloured $k$-enabling, with the $k$-clique of colour $i$ containing $(x,y)$ given by 
    \begin{equation*}
        C_i((x,y)) = \begin{cases}
            \{(x+in, y+n): n \in \mathbb{Z}_p\} &\text{if } i \in \mathbb{Z}_p \\
            \{(x+n, y): n \in \mathbb{Z}_p\} &\text{if } i = \infty.
        \end{cases}
    \end{equation*}
    Then $|V(G)| = p^2 < 2r(k-1) = 2p^2-2$, as required.
\end{proof}

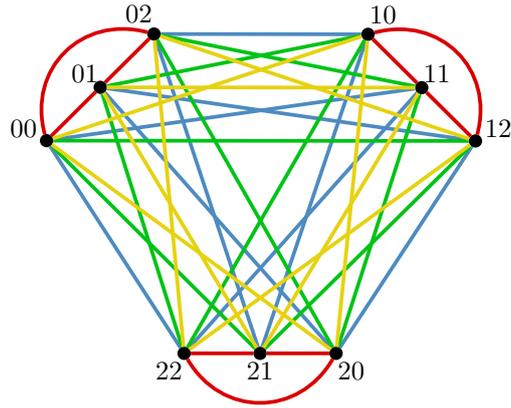
\begin{figure}[h]
    \centering
    \begin{tikzpicture}
    \begin{scope}[rotate=45]
        \node[mycircle, label={[xshift=-0.3cm, yshift=-0.2cm]\footnotesize{$00$}}] (1a) at (0,0){};
        \node[mycircle, label={[xshift=-0.2cm, yshift=-0.18cm]\footnotesize{$01$}}] (2a) at (1,0){};
        \node[mycircle, label={[xshift=-0.2cm, yshift=-0.1cm]\footnotesize{$02$}}] (3a) at (2,0){};
    \end{scope}

    \begin{scope}[xshift=4.23cm, yshift=1.41cm, rotate=-45]
        \node[mycircle, label={[xshift=0.2cm, yshift=-0.1cm]\footnotesize{$10$}}] (1b) at (0,0){};
            \node[mycircle, label={[xshift=0.2cm, yshift=-0.18cm]\footnotesize{$11$}}] (2b) at (1,0){};
            \node[mycircle, label={[xshift=0.3cm, yshift=-0.2cm]\footnotesize{$12$}}] (3b) at (2,0){};
    \end{scope}

    \begin{scope}[xshift=1.81cm, yshift=-2.82cm]
        \node[mycircle, label={[xshift=0.2cm, yshift=-0.6cm]\footnotesize{$20$}}] (1c) at (2,0){};
            \node[mycircle, label={[yshift=-0.6cm]\footnotesize{$21$}}] (2c) at (1,0){};
            \node[mycircle, label={[xshift=-0.2cm, yshift=-0.6cm]\footnotesize{$22$}}] (3c) at (0,0){};

    \end{scope}

    \foreach \x [evaluate=\x as \y using int(\x+1)] in {1,2} {
        \draw (\x a) edge[color=myred] (\y a);
    }
    \draw (1a) edge[color=myred, bend left=60, looseness=1.2] (3a);

    \foreach \x [evaluate=\x as \y using int(\x+1)] in {1,2} {
        \draw (\x b) edge[color=myred] (\y b);
    }
    \draw (1b) edge[color=myred, bend left=60, looseness=1.2] (3b);

    \foreach \x [evaluate=\x as \y using int(\x+1)] in {1,2} {
        \draw (\x c) edge[color=myred] (\y c);
    }
    \draw (1c) edge[color=myred, bend left=60, looseness=1.2] (3c);

    \draw (3a) edge[color=myblue] (1b);
    \draw (1b) edge[color=myblue] (2c);
    \draw (3a) edge[color=myblue] (2c);
    \draw (3b) edge[color=myblue] (1c);
    \draw (1c) edge[color=myblue] (2a);
    \draw (3b) edge[color=myblue] (2a);
    \draw (2b) edge[color=myblue] (3c);
    \draw (3c) edge[color=myblue] (1a);
    \draw (2b) edge[color=myblue] (1a);

    \draw (1a) edge[color=mygreen] (2c);
    \draw (3b) edge[color=mygreen] (2c);
    \draw (3b) edge[color=mygreen] (1a);
    \draw (2a) edge[color=mygreen] (3c);
    \draw (1b) edge[color=mygreen] (3c);
    \draw (1b) edge[color=mygreen] (2a);
    \draw (3a) edge[color=mygreen] (1c);
    \draw (2b) edge[color=mygreen] (1c);
    \draw (2b) edge[color=mygreen] (3a);

    \draw (1a) edge[color=myyellow] (1b);
    \draw (1b) edge[color=myyellow] (1c);
    \draw (1c) edge[color=myyellow] (1a);
    \draw (2a) edge[color=myyellow] (2b);
    \draw (2b) edge[color=myyellow] (2c);
    \draw (2c) edge[color=myyellow] (2a);
    \draw (3a) edge[color=myyellow] (3b);
    \draw (3b) edge[color=myyellow] (3c);
    \draw (3c) edge[color=myyellow] (3a);
\end{tikzpicture}
    \caption{Extremal construction for \cref{lem:example2} in the case $k=3$ (and $r=4$).}
    \label{fig:multi_colour_construction2}
\end{figure}

\printbibliography
\end{document}